\documentclass[12pt,reqno]{amsart}
\usepackage{amsmath,amsthm,amsfonts,amssymb,times}
\usepackage{verbatim}
\usepackage{url}
\setbox0=\hbox{$+$}
\newdimen\plusheight
\plusheight=\ht0
\def\+{\;\lower\plusheight\hbox{$+$}\;}

\setbox0=\hbox{$-$}
\newdimen\minusheight
\minusheight=\ht0
\def\-{\;\lower\minusheight\hbox{$-$}\;}

\setbox0=\hbox{$\cdots$}
\newdimen\cdotsheight
\cdotsheight=\plusheight
\def\cds{\lower\cdotsheight\hbox{$\cdots$}}

\makeatletter
\def\leqalignno#1{\displ@y \tabskip\z@ plus\@ne fil
  \halign to\displaywidth{\hfil$\@lign\displaystyle{##}$\tabskip\z@skip
    &$\@lign\displaystyle{{}##}$\hfil\tabskip\z@ plus\@ne fil
    &\kern-\displaywidth\rlap{$\@lign\hbox{\rm##}$}\tabskip\displaywidth\crcr
    #1\crcr}}
\makeatother

\newcommand{\eb}{\begin{equation}}
\newcommand{\ee}{\end{equation}}

 \renewcommand{\a}{\alpha}
\renewcommand{\b}{\beta}

\renewcommand{\k}{\kappa}

\renewcommand{\(}{\left\(}
\renewcommand{\)}{\right\)}
\renewcommand{\[}{\left\[}
\renewcommand{\]}{\right\]}

\numberwithin{equation}{section}
 \theoremstyle{plain}
\newtheorem{theorem}{Theorem}[section]
\newtheorem{lemma}[theorem]{Lemma}
\newtheorem{corollary}[theorem]{Corollary}

\newtheorem{remark}[theorem]{Remark}

\begin{document}
\title[On Sum-Of-Tails Identities]{On Sum-Of-Tails Identities} 

\author{Rajat Gupta }
\address{Department of Mathematics, Indian Institute of Technology Gandhinagar, Palaj, Gandhinagar 382355, Gujarat, India} 
\email{rajat\_gupta@iitgn.ac.in}
\maketitle

\begin{abstract}
In this article, a finite analogue of the generalized sum-of-tails identity of Andrews and Freitas is obtained. We derive several interesting results as special cases of this analogue, in particular, a recent identity of Dixit, Eyyyunni, Maji and Sood. We derive a new extension of Abel's lemma with the help of which we obtain a one-parameter generalization of a sum-of-tails identity of Andrews, Garvan and Liang, an identity of Ramanujan as well as two new results - one for Ramanujan's function $\sigma(q)$ and another for the function recently introduced by Andrews and Ballantine. Later we introduce a new generalization $\textup{FFW}_{c}(n)$ of a function of Fokkink, Fokkink and Wang and derive an identity for its generating function. This gives, as a special case, a recent representation for the generating function of $\textup{spt}(n)$ given by Andrews, Garvan and Liang. We also obtain some weighted partition identities along with new representations for two of Ramanujan's third order mock theta functions through combinatorial techniques.
\end{abstract}

\section{Introduction}\label{Introduction}
Ramanujan's Lost Notebook \cite[p.~14]{Ramanujannarosa} contains the following beautiful sum-of-tails identity:
\begin{align}\label{oldramanujan}
\sum_{n=0}^{\infty}[(-q;q)_\infty - (-q;q)_n ]=(-q,q)_\infty\left(-\frac{1}{2}+ \sum_{n=1}^{\infty}\frac{q^n}{1-q^n} \right)
+\frac{1}{2}\sum_{n=0}^{\infty}\frac{q^{n(n+1)/2}}{(-q;q)_n}.
\end{align}
Zagier \cite{zagier} derived the following identity of a similar type: $$\displaystyle{\sum_{n=1}^{\infty}\left( \eta(24z)-q(1-q^{24})(1-q^{48})....(1-q^{24n})\right)=\eta(24z)D(q)+ E(q)},$$ where $\displaystyle{\eta(z)=q^{1/24}\prod_{n=1}^{\infty}(1-q^n), D(q)=\frac{1}{2}-\sum_{n=1}^{\infty}\frac{q^n}{1-q^n}}$, $E(q)=\displaystyle{\frac{1}{2}\sum_{n=1}^{\infty}n \chi(n)q^{(n^2-1)/24}}$ and $\chi(n)$ is the unique quadratic primitive character of conductor $12$. In \cite[Theorem 1, 2]{andrewsono} G. Andrews, J. Jim$\acute{\text{e}}$nez-Urroz, and K. Ono, gave a general identity which sums the difference between the q-product and its truncation.

The motivation for this project stemmed from generalizing \cite[Equation (8.1)]{DEMS}, namely, 
\begin{align}\label{Untrodenidentity}
\sum_{n=1}^{\infty}(-1)^{n-1}\left[(q^n)_{N}-1\right]= \frac{1}{2}\left(\frac{(q)_N}{(-q)_N}-1 \right).
\end{align} 
When we let $N \to \infty$ in the above identity, we get the well-known identity
\begin{align}
\sum_{n=1}^{\infty}(-1)^{n-1}\left[(q^n)_{\infty}-1\right]= \frac{1}{2}\left(\frac{(q)_{\infty}}{(-q)_{\infty}}-1 \right).
\end{align}
In \cite[Equation (4)]{fuyan} Yan and Fu derived the following identity, namely, 
\begin{align}\label{yanfu}
\sum_{n=1}^{\infty}\left[\begin{matrix} N\\n\end{matrix}\right]\frac{(-1)^nq^{n(n+1)/2}}{(1-cq^n)} = \frac{1}{(1-c)}\left(\frac{(q)_N}{(cq)_N}-1 \right).
\end{align}
The special case of one of our Theorems \ref{a,t,g,N} below is, 
\begin{align}\label{halfcwalacase}
\sum_{n=1}^{\infty}c^{n-1}\left[(q^n)_{N}-1\right] = \sum_{n=1}^{\infty}\left[\begin{matrix} N\\n\end{matrix}\right]\frac{(-1)^nq^{n(n+1)/2}}{(1-cq^n)}.
\end{align} 
Taking $c =-1$ case in the identity obtained by equating the right-hand side of \eqref{yanfu} and the left-hand side of \eqref{halfcwalacase}, we get \eqref{Untrodenidentity}. Actually, the following is true:
\begin{align}\label{cwalacase}
\sum_{n=1}^{\infty}c^{n-1}\left[(q^n)_{N}-1\right] = \sum_{n=1}^{\infty}\left[\begin{matrix} N\\n\end{matrix}\right]\frac{(-1)^nq^{n(n+1)/2}}{(1-cq^n)} = \frac{1}{(1-c)}\left(\frac{(q)_N}{(cq)_N}-1 \right).
\end{align}
As a special case, letting $N \to \infty$, we get the following identity \cite[Theorem 3.5]{andrewsgarvanliang} of G. Andrews, F. Garvan and J. Liang which they proved  combinatorially by generalizing $\textup{FFW}(n)$ function, introduced in \cite{andrewsgarvanliang}, namely,
\begin{equation}\label{AGL}
\sum_{n=1}^{\infty}c^{n-1}\left[(q^n)_{\infty}-1\right] = \sum_{n=1}^{\infty}\frac{(-1)^nq^{n(n+1)/2}}{(q)_n(1-cq^n)} = \frac{1}{(1-c)}\left(\frac{(q)_\infty}{(cq)_\infty}-1 \right).
\end{equation}
We will discuss more about the $\textup{FFW}{(n)}$ function later in this section, where an interesting generalization of $\textup{FFW}{(n)}$ is obtained.  


Identity \eqref{AGL}, in turn, comes as a special case of one of the main results of this paper, namely, 
\begin{theorem}\label{a,t,g,N}
Let $N$ be positive integer, $\displaystyle{f(x):=\sum_{n=0}^{\infty}f_{n}x^{n}}$ and $\displaystyle{g(x):=\sum_{n=0}^{\infty}g_{n}x^{n}}$, then for $|q|<1, |t|<1,$  we have 
\begin{align}
\sum_{n=0}^{\infty}g_{n}\left[\frac{(aq^N)_{n}}{(tq^N)_{n}}\frac{(t)_{n}}{(a)_{n}} -\frac{(t)_{N}}{(a)_{N}}\right] = \frac{(t)_{N}}{(a)_{N}}\sum_{n=1}^{\infty}\left\{\sum_{k=0}^{n}\left[\begin{matrix} n\\k\end{matrix}\right]\left(\frac{aq^{N}}{t}\right)^k(q^{-N})_{k}(q^{N})_{n-k}\right\}\frac{g(q^n)}{(q)_n}t^n,
\end{align}
where $\displaystyle{\binom{n}{2}=\frac{n(n-1)}{2}}$ and $\left[\begin{matrix} n\\k\end{matrix}\right]= \displaystyle{\frac{(q)_{n}}{(q)_{n-k}(q)_{k}}}. $
\end{theorem}
This theorem also generalizes the result \cite[Theorem 4.1]{andrewsfreitas} of G. Andrews and P. Freitas, namely,
\begin{align}\label{andrewsfreitas1}
\sum_{n=0}^{\infty}g_{n}\left[\frac{(t)_n}{(a)_n}- \frac{(t)_\infty}{(a)_\infty} \right]= \frac{(t)_\infty}{(a)_\infty} \sum_{n=1}^{\infty}\frac{(a/t)_n}{(q)_n}g(q^n)t^n.
\end{align}
Indeed, letting $N\rightarrow \infty$ in Theorem \ref{a,t,g,N}, we get \eqref{andrewsfreitas1} as a special case. 

In this article, we aim to study more general sum-of-tails identities and their finite analogues. As an application of these sum-of-tails identities, we obtain generalizations of some useful results in basic hypergeometric series. We will also obtain interesting combinatorial interpretations of some of the theorems and corollaries derived in this article. Following is the notation used throughout the article.
\begin{itemize}
\item $\pi$: an integer partition,

\item $|\pi|$: sum of the parts of $\pi$,

\item $p(n)$: the number of integer partitions of $n$

\item $ s(\pi):=$ the smallest part of $\pi$,

\item $l(\pi):=$ the largest part of $\pi$,

\item $\#(\pi):=$ the number of parts of $\pi$,

\item $\mathrm{rank}(\pi)= l(\pi)- \#(\pi)$,

  

\item $\mathcal{P}(n):=$ collection of all integer partitions of $n$,

\item $\mathcal{D}(n):=$ collection of partitions of $n$ into distinct parts,

\item $\mathcal{D}_k(n):=$ collection of partitions of $n$ into distinct parts in which each part is greater than $k,$

\item $\mathcal{B}(n):=$ collection of partitions of $n$  in which only the smallest part is allowed to repeat,

\item $\mathcal{B'}(n):=$ collection of partitions of $n$  in which only the largest part is allowed to repeat.

\end{itemize}
We now give a special case of Theorem \ref{a,t,g,N}.

\begin{theorem}\label{generalheineth}
For $N \in \mathbb{N},$ $|b|<1$ and $|q|<1,$
\begin{align}\label{generalheine}
\sum_{n=0}^{\infty}\frac{(c/b)_n}{(q)_n} \frac{(t)_n}{(at)_n}\frac{(a tq^N)_{n}}{(tq^N)_{n}}b^n = \frac{(t)_{N}(c)_\infty}{(at)_{N}(b)_\infty}\sum_{n=0}^{\infty}\left\{\sum_{k=0}^{n}\left[\begin{matrix} n\\k\end{matrix}\right]\left(aq^{N}\right)^k(q^{-N})_{k}(q^{N})_{n-k}\right\}\frac{(b)_n}{(c)_n(q)_n}t^n.
\end{align}
\end{theorem}
On letting $N\rightarrow \infty$ in the above theorem we get Heine's $_2\phi_1$ transformation \cite[p. 359, Equation (III.$1$)]{Gasper}, namely,
\begin{equation}\label{heine}
\sum_{n=0}^{\infty}\frac{(a)_n(b)_n}{(c)_n(q)_n}t^n=\frac{(at)_{\infty}(b)_{\infty}}{(t)_{\infty}(c)_{\infty}}\sum_{n=0}^{\infty}\frac{(c/b)_n(t)_n}{(at)_n(q)_n}b^n.
\end{equation}
We now focus our attention on the second series that appeared on the  right of \eqref{oldramanujan}, that is,
$$\sigma(q):= \sum_{n=0}^{\infty}\frac{q^{n(n+1)/2}}{(-q;q)_{n}}.$$
Ramanujan himself gave another series representation for $\sigma(q)$ in \cite[p.~14]{Ramanujannarosa} which is,
\begin{align}
\sigma(q)=\sum_{n=0}^{\infty}\frac{q^{n(n+1)/2}}{(-q;q)_{n}}= 1+ \sum_{n=1}^{\infty}(-1)^{n-1}q^n(q)_{n-1}.
\end{align}
The function $\sigma(q)$ is not only important from the point of view of partitions \cite{andrewsV}, but also from the point of view of algebraic number theory \cite{ADH} and quantum modular forms \cite{zagier2}.

In \cite[Equation 4.6]{DEMS} A. Dixit, P. Eyyyunni, B. Maji and G. Sood studied a finite analogue of $\sigma(q),$ namely,
\begin{align*}
\sigma(q,N):=\sum_{n=0}^{\infty}\left[\begin{matrix} N\\n\end{matrix}\right]\frac{(q)_n q^{n(n+1)/2}}{(-q)_n},
\end{align*}
and gave its partition theoretic interpretation \cite[Corollary 4.3]{DEMS}. As a special case of Theorem \ref{generalheineth}, we are able to obtain  a new series presentation for $\sigma(q,N)$:


\begin{corollary}\label{finiteramanujansum}
For $N>0, ~~ |q|<1$
\begin{align*}
\sigma(q,N):=\sum_{n=0}^{\infty}\left[\begin{matrix} N\\n\end{matrix}\right]\frac{(q)_n q^{n(n+1)/2}}{(-q)_n}=\frac{(q)_{\infty}}{(-q^{N+1})_{\infty}}+2\sum_{n=1}^{\infty}\frac{q^n}{(1+q^n)}\frac{(q^{n+1})_\infty}{(-q^{N+n+1})_\infty}.
\end{align*}
\end{corollary}
Taking $N\to\infty$ case in the above identity, we get the following representation for $\sigma(q)$, which, to the best of our knowledge, seems to be new. It is, however, a  special case of \eqref{heine}.
\begin{equation}\label{ramaujansumsigma}
\sigma(q):=\sum_{n=0}^{\infty}\frac{q^{n(n+1)/2}}{(-q)_n}=(q)_{\infty}+ 2 \sum_{n=1}^{\infty}\frac{q^n}{1+q^n}(q^{n+1})_\infty.
\end{equation}
The above identity gives an interesting weighted partition identity.
\begin{corollary}\label{ramanujansumcombi}
Let $\mathcal{D}(n)$ denote the collection of partitions of $n$ into distinct parts and let $\mathcal{B}(n)$ denote the collection of partitions of $n$ in which only the smallest parts are allowed to repeat. Let $\#(\pi)$ denote the number of parts of a partition $\pi$. Then
\begin{equation}\label{1.14}
\sum_{\pi\in\mathcal{D}(n)}\left((-1)^{\#(\pi)}-(-1)^{\textup{rank}(\pi)}\right)=2\sum_{\pi\in\mathcal{B}(n)}(-1)^{\#(\pi)}.
\end{equation}
\end{corollary}
Another new representation for $\sigma(q)$ is obtained in Theorem \ref{newramanujan}$(i)$ below by deriving it as a special case of Lemma \ref{lemmanew} .
\begin{theorem}\label{newramanujan}
For $|q|<1$,
\begin{align*}
(i). &\sum_{n=1}^{\infty}\frac{q^n}{1-q^n}(-q^{n+1};q)_\infty=(-q;q)_\infty\left(\frac{1}{2}+ \sum_{n=1}^{\infty}\frac{q^n}{1+q^n} \right)-\frac{1}{2}\sum_{n=0}^{\infty}\frac{q^{n(n+1)/2}}{(-q;q)_n}.\\
(ii). &\sum_{n=1}^{\infty}\frac{q^{2n}}{1-q^{2n}}(-q^{2n+1};q^2)_\infty=(-q;q^2)_\infty\sum_{n=1}^{\infty}\frac{q^{2n-1}}{1+q^{2n-1}}-\sum_{n=1}^{\infty}\frac{q^{n^2}}{(-q;q^2)_n}.
\end{align*}
\end{theorem}
It can be noted from above that the right-hand side of $(i)$ is almost identical to that of \eqref{oldramanujan}. The second sum on the right-hand side of $(ii)$, which we define as 
\begin{align}\displaystyle{\delta(q):=\sum_{n=1}^{\infty}\frac{q^{n^2}}{(-q;q^2)_n}},
\end{align}
seems to be very interesting. All the  coefficients of this series upto  $q^{1000}$ are in absolute value less than $2,$  which hints that, possibly,  almost all of the coefficients of this series appear infinitely often and most of the terms are zero, similar to that in the case of $\sigma(q).$ Recently this series also occurred in \cite{AB} while studying the generating function of the total number of parts in all self conjugate partitions of a certain integer, where, indeed, the above property is said to follow from \cite[Theorem $5$]{ADH}. This motivates us to study this series. By using combinatorial techniques, we are able to find a beautiful identity for a generalization of $\delta(q)$. As a special case it gives another representation for $\delta(q)$. Moreover we obtain new representations for two  of the third order mock theta functions using our result (see Corollary \ref{mocktheta}).
\begin{theorem}\label{generaldeltafunction}
For $t \in \mathbb{R},~|q|<1,$
\begin{align}
\delta_{-t}(q):=\sum_{n=0}^{\infty}\frac{q^{n^2}}{(tq;q^2)_n}=1+ \sum_{n=1}^{\infty}t^{n-1}q^n(-q^2/t;q^2)_{n-1}.
\end{align}
\end{theorem}
If we put $t=-1$, we get 
\begin{equation}
\delta(q)=\sum_{n=0}^{\infty}\frac{q^{n^2}}{(-q;q^2)_n}=1+ \sum_{n=1}^{\infty}(-1)^{n-1}q^n(q^2;q^2)_{n-1}.
\end{equation}
As mentioned above, we get a new and elegant representation for third order mock theta functions $\phi(q)$ and $\psi(q)$ by putting $t=-q$ and $t=1$ in above theorem.
\begin{corollary}\label{mocktheta}
For $|q|<1,$
\begin{align*}
(i).&\quad  \phi(q)= \sum_{n=0}^{\infty}\frac{q^{n^2}}{(-q^2;q^2)_n}=1+ \sum_{n=1}^{\infty}(-1)^{n-1}q^{2n-1}(q;q^2)_{n-1},\\
(ii).&\quad  \psi(q)= \sum_{n=0}^{\infty}\frac{q^{n^2}}{(q;q^2)_n}=1+ \sum_{n=1}^{\infty}q^{n}(-q^2;q^2)_{n-1}.
\end{align*}
\end{corollary}

As another application of our Theorem \ref{a,t,g,N}, we obtain a finite analogue of an identity of G. Andrews and A. Freitas, namely \cite[Corollary 4.3(i)]{andrewsfreitas},
\begin{align}\label{andrewfreitas} 
\sum_{n=0}^{\infty}\left[(t)_n - (t)_\infty \right] = (t)_{\infty}\sum_{n=1}^{\infty}\frac{t^n}{(q)_{n}(1-q^n)},
\end{align}
along with an extra parameter $c.$
\begin{theorem}\label{theoremcwalacase2}
For $N \in \mathbb{N}$, $c \in \mathbb{C} ~such ~that~ c\neq q^{-n}, ~\forall n \in \mathbb{N}$, $|t|<1 ~and~$ $|q|
<1,$
\begin{align}\label{cwalacase2}
\sum_{n=0}^{\infty}c^n \left[ \frac{(t)_n}{(t)_{N+n}} - 1 \right] = \sum_{n=1}^{\infty}\frac{(q^N)_nt^n}{(q)_{n}(1-cq^n)}.
\end{align}
\end{theorem}
Letting $N \rightarrow \infty $, we have one parameter generalization of \cite[Corollary 4.3(i)]{andrewsfreitas}. For $c \in \mathbb{R}, |t|<1$
\begin{align}\label{oneparameterzagier1} 
\sum_{n=0}^{\infty} c^n\left[(t)_n - (t)_\infty \right] = (t)_{\infty}\sum_{n=1}^{\infty}\frac{t^n}{(q)_{n}(1-cq^n)}.
\end{align}
If we take $t =q$ and $c=1$ then we get \cite[Corollary 4.3(i)]{andrewsfreitas}. We will discuss more about the left-hand side of \eqref{oneparameterzagier1} in the next section along with its partition theoretic interpretation.

For $c \in \mathbb{R}$ it is easy to see, using induction, that
\begin{align}\label{zagier}
\sum_{n=0}^{N-1}c^n \left[ (t)_n - (t)_N \right] &= t\sum_{n=1}^{N}(1+c+c^2+...+c^n)(t)_{n-1}q^{n-1},\nonumber \\
&=\frac{t}{1-c}\sum_{n=1}^{N}(1-c^n)(t)_{n-1}q^{n-1},
\end{align}
which is a generalization of \cite[Equation $(16)$]{zagier}.

Now take the limit $N\rightarrow \infty$, which leads to 
\begin{align}\label{oneparameterzagier2}
\sum_{n=0}^{\infty}c^n \left[ (t)_n - (t)_\infty \right] &= \frac{t}{1-c}\sum_{n=1}^{\infty}(1-c^n)(t)_{n-1}q^{n-1}.
\end{align} 
\begin{remark}\label{remark1}
From \eqref{oneparameterzagier1} and \eqref{oneparameterzagier2}, for $c \in \mathbb{R}$ and $|t|< 1$ we have, 
\begin{align}\label{Remark1}
\sum_{n=1}^{\infty}\frac{t^n}{(q)_{n}(1-cq^n)} = \frac{t}{1-c}\sum_{n=1}^{\infty}\frac{(1-c^n)q^{n-1}}{(t q^{n-1})_{\infty}}.
\end{align}
Interesting special cases of \eqref{Remark1} are given in Section \ref{Section3}.
\end{remark}
In \cite[Theorem 3.4]{andrewsgarvanliang} G. Andrews, F.G. Garvan and J Liang, prove a sum-of-tails identity, namely, 
\begin{align}\label{AGL2}
\sum_{n=0}^{\infty}\frac{1}{(q)^2_n}\left( (q)_n -(q)_\infty \right) = \sum_{n=1}^{\infty}\frac{n q^{n^2}}{(q)^2_n}=q+\sum_{n=2}^{\infty}\sum_{m=1}^{n}mM(m,n)q^n,
\end{align}
where $M(m,n)$ is the number of partitions of $n$ with crank $m$. This motivates us to study a generalization of the left-hand side of \eqref{AGL2}. This is given in theorem below.
\begin{theorem}\label{genAGL}
For $b \in \mathbb{C}$, $|a|~and~|q|<1,$
\begin{align}\label{genAGLequation}
\sum_{n=0}^{\infty}\frac{1}{(\beta q)_n(q)_n}&\left( (\alpha)_n -(\alpha)_\infty \right)=\frac{(\alpha)_\infty}{(q)_\infty}\left[ \sum_{n=1}^{\infty}\frac{n \beta^n q^{n^2}}{(\beta q)_n(q)_n}+ \frac{1}{(\beta q)_\infty}\sum_{n=1}^{\infty} \frac{(\beta q/\alpha)_n}{1-q^n}\alpha^n \right].
\end{align}
\end{theorem}
If we consider $\alpha=q$ and  $\beta=1$ in \eqref{genAGLequation} then we obtain \eqref{AGL2} and $\alpha\to 0$ case of the above Theorem \ref{genAGL} leads to more general  \textit{half Lerch sum}. It is defined in \cite{xiong} by following series,
$$h(q):=\sum_{n=1}^{\infty} \frac{(-1)^nq^{n(n+1)/2}}{1-q^n}.$$ This is called half Lerch sum because its bilateral extension is almost a specialization of Lerch sum, namely, 
$$\sum_{n=-\infty}^{\infty}\frac{e^{\pi i(n^2+n)z+2\pi in v}}{1-e^{2\pi in z + 2\pi iu}}.$$
Further studies on these sums can be found in \cite{hm}, \cite{xiong}. Next we have some interesting corollaries of the above theorem.
\begin{corollary}\label{lerchsum}
For $\beta \in \mathbb{C}$ and $|q|<1,$ 
\begin{align}\label{lerchsumeqn}
h(\beta,q):=\sum_{n=1}^{\infty} \frac{(-\beta)^nq^{n(n+1)/2}}{1-q^n}=-(\beta q)_\infty\sum_{n=1}^{\infty}\frac{n \beta^n q^{n^2}}{(\beta q)_n(q)_n}.
\end{align}
\end{corollary}
This identity is given by Ramanujan \cite[p.~354]{Ramanujantifr} and \cite[p.~263, Entry 2]{berndt1}. The special case $\beta=1$ was rediscovered by Andrew, Chan and Kim \cite[Theorem 2]{andrewschankim}. Corollary \ref{lerchsum}, in turn, gives the following sum-of-tails identities.
\begin{corollary}\label{AGla=-q}
For $|q|<1$, 
\begin{align}\label{AGla=-qenq}
\sum_{n=0}^{\infty}\left( 1 -(-q^{n+1};q)_\infty \right)(q^{n+1};q)_\infty=\sum_{n=1}^{\infty}\frac{n (-1)^n q^{n^2}}{(q;q)_n}(-q^{n+1};q)_\infty=-\sum_{n=1}^{\infty} \frac{q^{n(n+1)/2}}{1-q^n}.
\end{align}
\end{corollary}

\begin{corollary}\label{AGla=q}
For $|q|<1$
\begin{align}\label{AGla=qeqn}
\sum_{n=0}^{\infty}\left( 1 -(q^{n+1};q)_\infty \right)(-q^{n+1};q)_\infty =2 \sum_{n=1}^{\infty}\frac{(-q)_nq^n}{1-q^{2n}}-\sum_{n=1}^{\infty} \frac{q^{n(n+1)/2}}{1-q^n}.
\end{align}
\end{corollary}
\begin{corollary}\label{subtraction}
For $|q|<1$
\begin{align}
\sum_{n=0}^{\infty}\left( (-q^{n+1};q)_\infty -(q^{n+1};q)_\infty\right)=2 \sum_{n=1}^{\infty}\frac{(-q)_nq^n}{1-q^{2n}}.
\end{align}
\end{corollary}

\begin{corollary}\label{AGLq->q^2case}
For $|q|<1$
\begin{align}\label{sumtail}
\sum_{n=0}^{\infty}\left( (q^{2n+2};q^2)_\infty -(q^{2n+1};q)_\infty \right)= (q;q^2)_\infty \sum_{n=1}^{\infty}\frac{n q^{n(2n-1)}}{(q;q)_{2n}} = \sum_{n=1}^{\infty}\frac{(-1)^{n-1}q^{n^2}}{1-q^{2n}}.
\end{align}
\end{corollary}

Fokkink, Fokkink and Wang \cite[Theorem 1]{FFW}  shows that if $\mathcal{D}_n$ is  the collection of partitions of $n$ into distinct parts and $s(\pi)$ denotes the smallest part in partition $\pi,$ then 
\begin{align}\label{FFWfunction}
\textup{FFW}(n):=\sum_{\pi \in \mathcal{D}_n}  (-1)^{\#(\pi)} s(\pi) = d(n),
\end{align}
where $d(n)
$ denotes the number of divisors of $n.$ Combinatorially $\textup{FFW}(n)$ is weighted sum over $\mathcal{D}_n$ with weight $(-1)^{\# (\pi)} s(\pi).$ 

Recently, G. Andrews, F. Garvan and J. Liang \cite[Equation (3.13)]{andrewsgarvanliang} generalized the \textup{FFW}$(n)$ function by a parameter $c$, and obtain identity \eqref{AGL}, namely,
\begin{align*}
\sum_{n=1}^{\infty}c^{n-1}\left[(q^n)_{\infty}-1\right] = \sum_{n=1}^{\infty}\frac{(-1)^nq^{n(n+1)/2}}{(q)_n(1-cq^n)} = \frac{1}{(1-c)}\left(\frac{(q)_\infty}{(cq)_\infty}-1 \right).
\end{align*}
Here we give another generalization of $\textup{FFW}(n)$ function by taking the weight of  partition enumerated by $\mathcal{D}_n$ to be $(-c)^{\#(\pi)} s(\pi),$ where $c \in \mathbb{C}$.
\begin{align}\label{defn}
\textup{FFW}_c(n):=\sum_{\pi \in \mathcal{D}_n}  (-c)^{\#(\pi)} s(\pi).
\end{align}
Note that $\textup{FFW}_1(n)=\textup{FFW}(n).$ Using \eqref{defn} we have the following theorem.
\begin{theorem}\label{FFWfunctiongen}
For $|q|<1, c\in \mathbb{C},$
\begin{align}\label{FFWfunctiongeneqn}
\sum_{n=0}^{\infty} \textup{FFW}_{c}(n) q^n = -\sum_{n=1}^{\infty}\frac{(-c)^{n}q^{n(n+1)/2}}{(q)_n(1-q^n)}= \left( \sum_{n=1}^{\infty} \frac{q^n}{1-q^n} -\sum_{n=1}^{\infty}\frac{(c)_n}{1-q^n}q^n\right).
\end{align}
\end{theorem}
Letting $c\to1$ in the above Theorem \ref{FFWfunctiongen}, we have \cite[p. 14, Equations (12.4), (12.42)]{Fine}
\begin{align}\label{series}
\sum_{n=0}^{\infty} \textup{FFW}(n) q^n=\sum_{n=1}^{\infty}\frac{(-1)^{n-1}q^{n(n+1)/2}}{(q)_n(1-q^n)}=  \sum_{n=1}^{\infty} \frac{q^n}{1-q^n}.
\end{align} 
If we differentiate the above identity \eqref{FFWfunctiongeneqn} with respect to $c$, take its limit $c\to 1$ and then divide both sides by $(q)_\infty$, we get the representation of the generating function of $\textup{spt}(n)$ given by G. Andrews \cite{andrewssmallestpart} and \cite[Theorem 3.8, Equation (3.28)]{andrewsgarvanliang}

\begin{corollary}
For $|q|<1$,
\begin{align}
 \frac{1}{(q)_\infty}\sum_{n=1}^{\infty}\frac{n(-1)^{n-1}q^{n(n+1)/2}}{(q)_n(1-q^n)} = \sum_{n=1}^{\infty}\frac{q^n}{(1-q^n)^2(q^{n+1})_\infty}=\sum_{n=1}^{\infty} \textup{spt}(n)q^n.
\end{align}
\end{corollary}
More general form of the above middle sum in \eqref{series}, namely, $$\displaystyle{\sum_{n=1}^{\infty}\frac{(-1)^{n-1}q^{n(n+1)/2}}{(q)_n(1-q^n)^k}},$$ had been studied by many authors. In \cite[Theorem 3]{Dilcher} Dilcher studied this series and gave its series expression in terms of Stirling numbers. Later in \cite[Lemma 2.2]{andrewscrippa} Andrews, Crippa and Simon gave the following expression for $k \geq 1$, 
\begin{align}\label{crippa}
\sum_{n=1}^{\infty}\frac{(-1)^{n-1}q^{n(n+1)/2}}{(q)_n(1-q^n)^k}=(q)_\infty\sum_{n=0}^{\infty}\frac{q^n}{(q)_n}\binom{k+n-1}{k}.
\end{align}
In equation \eqref{AGL}, Andrews, Garvan and  Liang gave the sum-of-tail identity for the  middle sum in \eqref{series}  A natural question to ask is, can we obtain a sum-of-tail identity for the left-hand side of \eqref{crippa}? We answer this question in the affirmative in the following theorem.
\begin{theorem}\label{crippawalacasewithacoro}
For $N \in \mathbb{N},k\geq1,  c \in \mathbb{C}$ and $|q|<1$
\begin{align}\label{crippawalacasewitha}
\sum_{n=0}^{\infty}c^n \binom{k+n-1}{n}( {(aq^n)_{N}} - 1)=\sum_{n=1}^{\infty}\frac{(-a)^n q^{n(n-1)/2}}{(q)_n(1-c q^n)^k}(q^{N-n+1})_n.
\end{align}
\end{theorem}
Letting $N \to \infty$ with $a = q$ we have 
\begin{align}\label{crippawalacase}
\sum_{n=1}^{\infty}c^n \binom{k+n-1}{n}\left\{1 - (q^{n+1})_\infty\right\}=\sum_{n=1}^{\infty}\frac{(-1)^{n-1} q^{n(n+1)/2}}{(q)_n(1-c q^n)^k}.
\end{align}
\begin{remark}
If we compare \eqref{crippa} and \eqref{crippawalacase} with $c=1$ we get simple looking sum of tail identity, 
\begin{align*}
\sum_{n=0}^{\infty}\binom{k+n-1}{n}\left\{\frac{1}{(q)_\infty} - \frac{1}{(q)_{n}}\right\}=\sum_{n=0}^{\infty}\binom{k+n-1}{k}\frac{q^n}{(q)_n}.
\end{align*}
\end{remark}

\section{Preliminaries}
From \cite[Equation (3.3.5), (3.3.7)]{andrews} we have
\begin{align}\label{andrews1}
(x)_N&= \sum_{j=0}^{N}\left[\begin{matrix} N\\j\end{matrix}\right]_{q}(-1)^jx^jq^{j(j-1)/2},
\end{align}
\begin{align}\label{andrews2}
\frac{1}{(x)_N}&=\sum_{j=0}^{\infty}\left[\begin{matrix} N+j-1\\j\end{matrix}\right]_{q}x^j.
\end{align}
From \cite[Equation (20)]{Gasper} we have for $\a \in \mathbb{C}$ and $|q|<1,$
\begin{align}\label{qanalogue}
\sum_{n=1}^{\infty}\frac{(\a)_n}{(q)_n}z^n=\frac{(az)_\infty}{(z)_\infty}.
\end{align}

Ramanujan's $_1\psi_1~$ summation formula \cite[p.~138, Equation (5.2.1)]{Gasper} states that for   $|\b/\a|<|z|<1$ and $|q|<1,$
\begin{align}\label{ramanujan1psi1}
\sum_{n=-\infty}^{\infty}\frac{(\alpha)_n}{(\beta)_{n}}z^n= \frac{(\alpha z)_\infty (q/ \alpha z)_\infty (q)_\infty (\beta/ \alpha )_\infty}{ ( z)_\infty (\beta/ \alpha z)_\infty (\beta)_\infty (q/ \alpha )_\infty}.
\end{align}
We will also use a result first observed in \cite[Lemma 2.2]{andrewsfreitas}. 
\begin{lemma}\label{basiclemma}
Let $f$ and $g$ be two functions given by 
$$f(x)= \sum_{n=0}^{\infty}f_n x^n \quad and \quad g(x)= \sum_{n=0}^{\infty}g_n x^n.$$ Assume that these series and $$\displaystyle{\sum_{n=0}^{\infty}\sum_{k=0}^{\infty}|g_nf_kq^{nk}x^k|},$$\\ converge absolutely. Then
\begin{align*}
\sum_{n=0}^{\infty}f_ng(q^n)x^n = \sum_{n=0}^{\infty}g_nf(q^n x).
\end{align*}
\end{lemma}
\section{Proofs of Theorems and its corollaries}\label{Section3}

\begin{proof}[Proof of Theorem \textup{\ref{a,t,g,N}}]
We know that  
\begin{align}\label{product of two series}
\left(\sum_{n=0}^{\infty}a_n x^n\right)\left(\sum_{m=0}^{\infty}b_{n} x^m\right)= \sum_{n=0}^{\infty}\left(\sum_{k=0}^{n}a_{k}b_{n-k}\right)x^n.
\end{align}
Take the left-hand side of \eqref{andrews1} with $x$ replaced by $bqx$  and the left-hand side of  \eqref{andrews2} with $x$ replaced by $qx$, multiply the resluting expressions, then use  \eqref{product of two series} in second step with the fact $\left[\begin{matrix} N\\n\end{matrix}\right] = 0$ for $n > N,$ to see that 
\begin{align*} 
\frac{(bqx)_N}{(qx)_N}&=\sum_{n=0}^{\infty}\left[\begin{matrix} N\\n\end{matrix}\right]_{q}(-1)^n(bx)^nq^{n(n+1)/2} \times \sum_{n=0}^{\infty}\left[\begin{matrix} N+n-1\\n\end{matrix}\right](qx)^n \nonumber \\
&=\sum_{n=0}^{\infty}\left( \sum_{k=0}^{n} \left[\begin{matrix} N\\k\end{matrix}\right]_{q}(-1)^kb^kq^{k(k+1)/2}\left[\begin{matrix} N+n-k-1\\n-k\end{matrix}\right]q^{n-k}\right)x^n \nonumber \\
&=\sum_{n=0}^{\infty}\left( \sum_{k=0}^{n} \left[\begin{matrix} N\\k\end{matrix}\right]\left[\begin{matrix} N+n-k-1\\n-k\end{matrix}\right](-1)^kb^kq^nq^{k(k-1)/2}\right)x^n.\nonumber 
\end{align*}
Thus, 
\begin{align*}
\frac{(bqx)_N}{(qx)_N}-1&=\sum_{n=1}^{\infty}\left( \sum_{k=0}^{n} \left[\begin{matrix} N\\k\end{matrix}\right]\left[\begin{matrix} N+n-k-1\\n-k\end{matrix}\right]_{q}(-1)^kb^kq^nq^{k(k-1)/2}\right)x^n.
\end{align*}
Define
\begin{align*}
f(x):=\frac{(bqx)_N}{(qx)_N}-1=\sum_{n=1}^{\infty}\left( \sum_{k=0}^{n} \left[\begin{matrix} N\\k\end{matrix}\right]_{q}\left[\begin{matrix} N+n-k-1\\n-k\end{matrix}\right]_{q}(-1)^kb^kq^nq^{k(k-1)/2}\right)x^n.
\end{align*}
Then
\begin{align}
f(q^nx)=\frac{(bq^{n+1}x)_N}{(q^{n+1}x)_N}-1=\frac{(bqx)_N}{(qx)_N}\left[\frac{(bq^{N+1}x)_n}{(bqx)_n}\frac{(qx)_n}{(q^{N+1}x)_n} - \frac{(qx)_N}{(bqx)_N}\right],
\end{align}
where we used \cite[Equation (1.2.34)]{Gasper}
\begin{align}\label{gasper1}
(xq^n)_N=\frac{(x)_N(xq^N)_n}{(x)_n}.
\end{align}

Let $qx =t$, $bqx=a$ and $\displaystyle{g(x)= \sum_{n=0}^{\infty}g_nx^n} $ such that this series and $\displaystyle{\sum_{n=0}^{\infty}\sum_{k=0}^{\infty}|g_nf_kq^{nk}x^k|},$ converge absolutely. Then we have
\begin{align}\label{ooperwala}
\sum_{n=0}^{\infty}f_ng(q^n)x^n &= \sum_{n=0}^{\infty}\left( \sum_{k=0}^{n} \left[\begin{matrix} N\\k\end{matrix}\right]_{q}\left[\begin{matrix} N+n-k-1\\n-k\end{matrix}\right]_{q}(-1)^kb^kq^nq^{k(k-1)/2}\right)g(q^n)x^n \nonumber \\
&=\sum_{n=0}^{\infty}\left( \sum_{k=0}^{n} \left[\begin{matrix} N\\k\end{matrix}\right]_{q}\left[\begin{matrix} N+n-k-1\\n-k\end{matrix}\right]_{q}\left(-\frac{a}{t}\right)^kq^{k(k-1)/2}\right)g(q^n)t^n.
\end{align}
Upon using the fact for $x \in \mathbb{C}$, $k \in \mathbb{N}$ and $|q|<1,$
\begin{align}\label{gasper3}
\left[\begin{matrix} x\\k\end{matrix}\right]&=\frac{(q^{-x})_k}{(q)_k}(-q^x)^kq^{- \binom{k}{2}},\\
\left[\begin{matrix} k+x\\k\end{matrix}\right]&=\frac{(q^{x+1})_k}{(q)_k},
\end{align} 
equation \eqref{ooperwala} becomes
\begin{align*}
\sum_{n=0}^{\infty}f_ng(q^n)x^n=\sum_{n=0}^{\infty}\left( \sum_{k=0}^{n}      \left[\begin{matrix} n\\k\end{matrix}\right](q^{-N})_{k}(q^{N})_{n-k}\left(\frac{aq^{N}}{t}\right)^k\right)\frac{g(q^n)}{(q)_n}t^n.
\end{align*} 
Also,
\begin{align*}
\sum_{n=0}^{\infty}g_nf(q^nx) &=\sum_{n=0}^{\infty}g_n\frac{(bqx)_N}{(qx)_N}\left[\frac{(bq^{N+1}x)_n}{(bqx)_n}\frac{(qx)_n}{(q^{N+1}x)_n} - \frac{(qx)_N}{(bqx)_N}\right]\\
&=\frac{(a)_N}{(t)_N}\sum_{n=0}^{\infty}g_n\left[\frac{(aq^{N})_n}{(a)_n}\frac{(t)_n}{(tq^{N})_n} - \frac{(t)_N}{(a)_N}\right].
\end{align*}
Theorem \ref{a,t,g,N} now follows from Lemma \ref{basiclemma}.
\end{proof}

\begin{proof}[Proof of Theorem \textup{\ref{generalheineth}}]
Upon letting  
\begin{equation}\label{qbinomial}
g(x)=\sum_{n=1}^{\infty}\frac{(c/b)_n}{(q)_n}(b x)^n = \frac{(cx)_\infty}{(bx)_\infty}, \qquad  |b x|<1,
\end{equation}
and replacing $a$ with $at,$ the left-hand side of Theorem \ref{a,t,g,N} becomes 
\begin{align}\label{first}
\sum_{n=0}^{\infty}\frac{(c/b)_n}{(q)_n}\left[\frac{(a tq^N)_{n}}{(tq^N)_{n}}\frac{(t)_{n}}{(at)_{n}} -\frac{(t)_{N}}{(at)_{N}}\right]b^n = \sum_{n=0}^{\infty}\frac{(c/b)_n}{(q)_n} \frac{(t)_n}{(at)_n}\frac{(a tq^N)_{n}}{(tq^N)_{n}}b^n -\frac{(t)_N}{(a t)_N}\frac{(c)_\infty}{(b)_\infty},
\end{align}
where we used \eqref{qbinomial} with $x=1$. 

The right-hand side of Theorem \ref{a,t,g,N} results in
\begin{align}\label{second}
\frac{(t)_{N}}{(at)_{N}}&\sum_{n=1}^{\infty}\left\{\sum_{k=0}^{n}\left[\begin{matrix} n\\k\end{matrix}\right]\left(aq^{N}\right)^k(q^{-N})_{k}(q^{N})_{n-k}\right\}\frac{(cq^n)_\infty}{(bq^n)_\infty(q)_n}t^n \nonumber\\
&=-\frac{(t)_N}{(a t)_N}\frac{(c)_\infty}{(b)_\infty} + \frac{(t)_{N}(c)_\infty}{(at)_{N}(b)_\infty}\sum_{n=0}^{\infty}\left\{\sum_{k=0}^{n}\left[\begin{matrix} n\\k\end{matrix}\right]\left(aq^{N}\right)^k(q^{-N})_{k}(q^{N})_{n-k}\right\}\frac{(b)_n}{(c)_n(q)_n}t^n.
\end{align}
Hence from \eqref{first} and \eqref{second}, 
\begin{align}
\sum_{n=0}^{\infty}\frac{(c/b)_n}{(q)_n} \frac{(t)_n}{(at)_n}\frac{(a tq^N)_{n}}{(tq^N)_{n}}b^n = \frac{(t)_{N}(c)_\infty}{(at)_{N}(b)_\infty}\sum_{n=0}^{\infty}\left\{\sum_{k=0}^{n}\left[\begin{matrix} n\\k\end{matrix}\right]\left(aq^{N}\right)^k(q^{-N})_{k}(q^{N})_{n-k}\right\}\frac{(b)_n}{(c)_n(q)_n}t^n.
\end{align}
This completes the proof.
\end{proof}
We now prove Heine's transformation \eqref{heine}

Letting $N\to\infty$  followed by  the fact 
\begin{align}\label{fact1}
\lim_{N\to \infty}(q^{-N})_k(q^N)^k &= \lim_{N\to \infty}(-1)^k q^{k(k-1)/2}(q^{N-k+1})_k \nonumber \\
& =(-1)^kq^{k(k-1)/2}
\end{align}
in the above identity to get \eqref{heine}. 

\begin{proof}[Proof of Corollary \textup{\ref{finiteramanujansum}}]
Replace $a \rightarrow a/t$ in \eqref{generalheine} and then take $t \rightarrow 0$ and $a=d q$. This gives 
\begin{align*}
\sum_{n=0}^{\infty}\frac{(c/b)_n}{(q)_n(dq)_n}(dq^{N+1})_{n}b^n = \frac{(c)_\infty}{(dq)_{N}(b)_\infty}\sum_{n=0}^{\infty}\left(dq^{N+1}\right)^n(q^{-N})_{n}\frac{(b)_n}{(c)_n(q)_n}.
\end{align*}
Assuming $b =q$, $c=-q, d=-1$ and using the fact that $$\displaystyle{\left(q^{N+1}\right)^n(q^{-N})_{n}=(-1)^nq^{n(n+1)/2}\frac{(q)_N}{(q)_{N-n}}},$$ we have 
\begin{align*}
\sum_{n=0}^{\infty}\frac{(-1)_n}{(q^2;q^2)_n}q^n(-q^{N+1})_{n}&= \frac{(-q^{N+1})_\infty}{(q)_\infty}\sum_{n=0}^{\infty}\left[\begin{matrix} N\\n\end{matrix}\right]\frac{(q)_nq^{n(n+1)/2}}{(-q)_n}.
\end{align*}
Hence
\begin{align*}
\frac{(q)_\infty}{(-q^{N+1})_\infty}+\frac{(q)_\infty}{(-q^{N+1})_\infty}\sum_{n=1}^{\infty}\frac{(-1)_n}{(q^2;q^2)_n}q^n(-q^{N+1})_{n}&=\sum_{n=0}^{\infty}\left[\begin{matrix} N\\n\end{matrix}\right]\frac{(q)_nq^{n(n+1)/2}}{(-q)_n}.
\end{align*}
Corollary \ref{finiteramanujansum} follows upon simplification.
\end{proof}
We now state the Abel-type lemma which played a crucial rule in proving the results in \cite{andrewsfreitas}. 
\begin{lemma}\label{Abel}
Suppose that
$$f_\alpha(z)= \sum_{n=0}^{\infty}\alpha_n z^n$$ is analytic for $|z|<1.$ and $\alpha \in \mathbb{C}$ for which \\
(i) \begin{align*}
\sum_{n=0}^{\infty}(\alpha -\alpha_n) < \infty, 
\end{align*}
(ii) \begin{align*}
\lim_{n \rightarrow \infty} n(\alpha -\alpha_n)=0.
\end{align*}
Then 
$$\lim_{z \rightarrow 1^{-}}\frac{d}{dz}(1-z)f_\alpha(z)= \sum_{n=0}^{\infty}(\alpha -\alpha_n).$$ 
\end{lemma}
We now use this lemma to derive another lemma which helps to prove our Theorem \ref{newramanujan}. As a special case we obtain \cite[Theorem 1]{patkowski}.
\begin{lemma}\label{lemmanew}
Assume that for $r\geq 1,$ $\left(a^{(i)}_{n}, a^{(i)}\right)_{i=1}^{r}$ and $\left(g_{n}, g\right)$ be the pairs which satisfy the hypothesis of Lemma \ref{Abel}. Define $\displaystyle{\Omega:=\prod_{i=1}^{r}a^{(i)}}$, $\displaystyle{f_{a}(z):= \sum_{n=0}^{\infty}a_{n}z^n}$, $\displaystyle{f_{g}(z):= \sum_{n=0}^{\infty}g_{n}z^n}.$ Then
\begin{align}
\sum_{n=0}^{\infty}&g_n\prod_{i=1}^{r}\left(a^{(i)}_{n}-a^{(i)}\right)\nonumber \\
&=\lim_{z\rightarrow 1^-}\frac{d}{dz}(1-z)\Bigg\{- f_{\Omega g}(z)+ \sum_{1 \leq i_1 \leq r}a^{(i_1)}f_{\tfrac{\Omega}{a^{(i_1)}}g}(z)  - \sum_{1\leq i_1<i_2 \leq r}a^{(i_1)}a^{(i_2)}f_{\tfrac{\Omega}{a^{(i_1)}a^{(i_1)}}g}(z)+ ....\nonumber\\
&....+ (-1)^{r}\sum_{1\leq i_1<i_2 <..<i_{r-2}<i_{r-1} \leq r} a^{(i_1)}a^{(i_2)}..a^{(i_{r-1})} f_{\frac{\Omega}{a^{(i_1)}a^{(i_2)}..a^{(i_{r-1})}}g}(z) + (-1)^{r+1} \Omega f_{g}(z) \Bigg\}.
\end{align}
Case $r=2$, 
\begin{align}\label{case2}
\sum_{n=0}^{\infty}&g_n\left(a_{n}-a\right)\left(b_{n}-b\right)\nonumber \\
&=\lim_{z\rightarrow 1^-}\frac{d}{dz}(1-z)\left\{-f_{a b g}(z)+( b f_{a  g}(z)+ a f_{b g}(z))-a b f_{g}(z)\right\}.
\end{align}
\end{lemma}

\begin{proof}
The proof easily follows from induction.
\end{proof}

\begin{proof}[Proof of Theorem \textup{\ref{newramanujan}}]
To prove the part $(i)$ , we consider \eqref{case2}
with $g_n=g=1.$ Then 
$$f_{g}(z)= \sum_{n=0}^{\infty}z^n=\frac{1}{1-z},$$ Hence
$$\frac{d}{dz}(1-z)f_g(z)=\frac{d}{dz}\frac{(1-z)}{(1-z)}=0.$$
From \eqref{case2} we obtain
\begin{align}\label{case2changed}
\sum_{n=0}^{\infty}&\left(a_{n}-a\right)\left(b_{n}-b\right)=\lim_{z\rightarrow 1^-}\frac{d}{dz}(1-z)\left\{-f_{a b }(z)+( b f_{a }(z)+ a f_{b}(z))\right\}.
\end{align}
Now consider $$a_n=\frac{(\alpha)_n}{(\beta)_n},~ b_n=\frac{(\gamma)_n}{(q)_n}~~with~~a=\frac{(\alpha)_\infty}{(\beta)_\infty},~ b=\frac{(\gamma)_\infty}{(q)_\infty}.$$

Let us first start with right-hand side of 
\eqref{case2changed}. Using Heine's $_2\phi_1$ transformation \eqref{heine}, 
we have
\begin{align*}
(1-z)f_{a b}(z)=(1-z)\sum_{n=0}^{\infty}\frac{(\alpha)_n(\gamma)_n}{(\beta)_n(q)_n}z^n=\frac{(\alpha z)_\infty (\gamma)_\infty}{(zq)_\infty(\beta)_\infty}\sum_{n=0}^{\infty}\frac{(\beta / \gamma)_n(z)_n}{(\alpha z)_n (q)_n}\gamma^n.
\end{align*}
Then
\begin{align}\label{newramanujanequation3}
\lim_{z\to 1^-}\frac{d}{dz}(1-z)f_{a b}(z)&=\frac{(\alpha)_\infty (\gamma)_\infty}{(q)_\infty(\beta)_\infty}\left(-\sum_{n=0}^{\infty}\frac{\alpha q^n}{1-\alpha q^n} + \sum_{n=1}^{\infty}\frac{q^n}{1-q^n} - \sum_{n=1}^{\infty}\frac{(\beta /\gamma)_n}{(1-q^n)(\alpha)_n}\gamma^n \right),
\end{align}
where we have used
\begin{align}\label{zn}
\left[\frac{d}{dz}(z;q)_n\right]_{z=1}=-(q;q)_{n-1},
\end{align}
\begin{align}\label{1byzq}
\left[\frac{d}{dz}\frac{1}{(z q;q)_\infty}\right]_{z=1}=\frac{1}{(q;q)_\infty}\sum_{n=1}^{\infty}\frac{q^n}{1-q^n},
\end{align}
\begin{align}\label{1byalphazq}
\left[\frac{d}{dz}\frac{1}{(\alpha z;q)_\infty}\right]_{z=1}=\frac{1}{(\alpha;q)_\infty}\sum_{n=0}^{\infty}\frac{\alpha q^n}{1-\alpha q^n}.
\end{align}
By Ramanujan's $_1\psi_1$ summation formula \eqref{ramanujan1psi1}, we have
\begin{align*}
f_{a}(z)=\sum_{n=0}^{\infty}\frac{(\alpha)_n}{(\beta)_n}z^n =\frac{(\alpha z)_\infty (q/ \alpha z)_\infty (q)_\infty (\beta/ \alpha )_\infty}{ ( z)_\infty (\beta/ \alpha z)_\infty (\beta)_\infty (q/ \alpha )_\infty} - \sum_{n=1}^{\infty}\frac{(q/\beta)_n}{(q/ \alpha)_n}(\beta / \alpha z)^n.
\end{align*}
Then upon simplification, we obtain
\begin{align}\label{newramanujanequation4}
\lim_{z\to 1^-}\frac{d}{dz}(1-z)f_{a}(z)=\frac{(\alpha)_\infty }{(\beta)_\infty}\Bigg( \sum_{n=1}^{\infty}\frac{q^n}{\alpha - q^n}+ &\sum_{n=1}^{\infty}\frac{q^n}{1-q^n} -\sum_{n=0}^{\infty}\frac{\alpha q^n}{1-\alpha q^n} -\sum_{n=0}^{\infty}\frac{\beta q^n}{\alpha-\beta q^n}\Bigg)\nonumber\\
&\qquad \qquad \qquad+\frac{(\alpha)_\infty}{(\beta)_{\infty}}\sum_{n=1}^{\infty}\frac{(q/\beta)_n}{(q/ \alpha)_n}(\beta / \alpha )^n.
\end{align}
Using $q$-analogue of binomial theorem \eqref{qanalogue}, we have
\begin{align*}
\frac{d}{dz}(1-z)f_{b}(z)=\frac{d}{dz}(1-z)\sum_{n=0}^{\infty}\frac{(\gamma)_n}{(q)_n}z^n=\frac{d}{dz}\frac{(\gamma z)_\infty}{(qz)_\infty}.
\end{align*}
Thus
\begin{align}\label{newramanujanequation5}
\lim_{z\to 1^-}\frac{d}{dz}(1-z)f_{b}(z)=\frac{(\gamma)_\infty}{(q)_\infty}\left(\sum_{n=1}^{\infty}\frac{q^n}{1-q^n}-\sum_{n=0}^{\infty}\frac{\gamma q^n}{1- \gamma q^n}\right).
\end{align}
Using \eqref{newramanujanequation3}, \eqref{newramanujanequation4} and \eqref{newramanujanequation5} in \eqref{case2changed}, we get
\begin{align}\label{newramanujanequation6}
\sum_{n=0}^{\infty}&\left(\frac{(\alpha)_n}{(\beta)_n}-\frac{(\alpha)_\infty}{(\beta)_\infty}\right)\left(\frac{(\gamma)_n}{(q)_n}-\frac{(\gamma)_\infty}{(q)_\infty}\right)\nonumber\\
&=\frac{(\alpha)_\infty (\gamma)_\infty}{(q)_\infty(\beta)_\infty}\Bigg(\sum_{n=1}^{\infty}\frac{(\beta /\gamma)_n}{(1-q^n)(\alpha)_n}\gamma^n+ \sum_{n=1}^{\infty}\frac{q^n}{\alpha - q^n} -\sum_{n=0}^{\infty}\frac{\beta q^n}{\alpha-\beta q^n}+\sum_{n=1}^{\infty}\frac{q^n}{1-q^n}-\sum_{n=0}^{\infty}\frac{\gamma q^n}{1- \gamma q^n}\Bigg)\nonumber \\
& \qquad \qquad \qquad \qquad \qquad \qquad +\frac{(\gamma)_\infty}{(q)_\infty}\sum_{n=1}^{\infty}\frac{(q/\beta)_n}{(q/ \alpha)_n}(\beta / \alpha )^n.
\end{align}
Taking $\gamma =q,~\alpha=-q$ and letting $\beta \to 0$ in above \eqref{newramanujanequation6}, we obtain
\begin{align*}
0=(-q)_\infty\Bigg(\sum_{n=1}^{\infty}\frac{q^n}{(1-q^n)(-q)_n}- &\sum_{n=0}^{\infty}\frac{q^n}{1 + q^n} \Bigg) + \sum_{n=1}^{\infty}\frac{q^{n(n-1)/2}}{(-1)_n}.
\end{align*}
Hence
\begin{align*}
\frac{1}{2}\sum_{n=0}^{\infty}\frac{q^{n(n+1)/2}}{(-q)_n}+(-q)_\infty\sum_{n=0}^{\infty}\frac{q^n}{1 + q^n}&=(-q)_\infty\sum_{n=1}^{\infty}\frac{q^n}{(1-q^n)(-q)_n}.
\end{align*}
The proof of part $(i)$ of the theorem is complete. 

To prove part $(ii)$  of the theorem, replace $q \to q^2$  in \eqref{newramanujanequation6} and then take $\alpha =-q, \beta \to 0$ and $\gamma =q^2.$ The proof is along the similar lines as that of part $(i)$, hence we omit the details.

%
\end{proof}

\begin{proof}[Proof of Theorem \textup{\ref{theoremcwalacase2}}]
In Theorem \ref{a,t,g,N} put $a=0$ and
$$g(x)= \sum_{n=0}^{\infty}c^n x^n =\frac{1}{1-cx} \qquad |c x|<1.$$
From the left-hand side of Theorem \ref{a,t,g,N}, we get
\begin{align*}
\sum_{n=0}^{\infty}g_n\left[\frac{(aq^{N})_n}{(a)_n}\frac{(t)_n}{(tq^{N})_n} - \frac{(t)_N}{(a)_N}\right]&= \sum_{n=0}^{\infty}c^n\left[\frac{(t)_n}{(tq^{N})_n} -(t)_N\right]  \\
&=(t)_N\sum_{n=0}^{\infty}c^n\left[\frac{(t)_n}{(t)_{N+n}} -1\right].
\end{align*}
Also, for $|cq|<1$,  the right-hand side of Theorem  \ref{a,t,g,N} becomes 
\begin{align*}
\frac{(t)_{N}}{(a)_{N}}&\sum_{n=1}^{\infty}\left\{\sum_{k=0}^{n}\left[\begin{matrix} n\\k\end{matrix}\right]\left(\frac{aq^{N}}{t}\right)^k(q^{-N})_{k}(q^{N})_{n-k}\right\}\frac{g(q^n)}{(q)_n}t^n \\
&=(t)_{N}\sum_{n=1}^{\infty}\frac{(q^N)_n}{(1-c q^n)(q)_n}t^n.
\end{align*}
This proves \eqref{cwalacase2} for $|c|<1/|q|.$ The results now follows for  $c \neq q^{-n}, ~\forall n \in \mathbb{N},$ by analytic continuation.

\end{proof}
\subsection*{Special cases of Remark \ref{remark1}:}
\begin{theorem} For $|t|<1$ and $|q|<1,$ we have
\begin{align}
(a).&~ \frac{1}{(t)_{\infty}} =1+ \sum_{n=1}^{\infty}\frac{tq^{n-1}}{(tq^{n-1})_{\infty}}.\\
(b).&~ \sum_{n=1}^{\infty}\frac{q^n}{(q)_{n-1}(1-q^n)^2} = \sum_{n=1}^{\infty}\frac{nq^{n}}{(q^n)_{\infty}}.\\
(c).&~\sum_{n=1}^{\infty}\frac{q^n}{(q)_{n}(1+q^n)} = \sum_{n=1}^{\infty}\frac{q^{2n-1}}{(q^{2n-1})_{\infty}}.\\
(d).&~\sum_{n=1}^{\infty}\frac{(-1)^{n+1}q^n}{(q)_{n}(1-q^n)} = \sum_{n=1}^{\infty}\frac{nq^{n}}{(-q^n)_{\infty}}.\\
(e).&~\sum_{n=1}^{\infty}\frac{(-1)^{n+1}q^n}{(q)_{n}(1+q^n)} = \sum_{n=1}^{\infty}\frac{q^{2n-1}}{(-q^{2n-1})_{\infty}}.
\end{align}
\end{theorem}

\begin{proof}

(a) followed by letting $c=0$ in \eqref{Remark1}, (b) followed by letting $t=q$ and $c=1$ in \eqref{Remark1}, (c) followed by letting $t=q$ and $c=-1$ in \eqref{Remark1}, (d) followed by letting $t=-q$ and $c=1$ in \eqref{Remark1}, (e) followed by letting $t=-q$ and $c=-1$ in \eqref{Remark1} 
\end{proof}



If $\textup{lpt}(n)$ denotes the total number of appearances of largest parts in all partitions of $n$, then we have the following result.

\begin{corollary}
For $n \in \mathbb{N},$
$$\textup{lpt}(n)=\textup{t}(n),$$
where $\textup{t}(n):=$ \textit{sum of the smallest part (without multiplicity) in all partitions of $n$.}

\end{corollary}
\begin{proof}
By simple combinatorics we can see that, 
$$\sum_{n=1}^{\infty}\frac{q^n}{(q)_{n-1}(1-q^n)^2} = \sum_{n=1}^{\infty} \textup{lpt}(n)q^n.$$ 
From (3.22) we know that, 
$$\sum_{n=1}^{\infty}\frac{q^n}{(q)_{n-1}(1-q^n)^2} = \sum_{n=1}^{\infty}\frac{nq^{n}}{(q^n)_{\infty}} = \sum_{n=1}^{\infty}\textup{t}(n)q^n.$$
Hence corollary follows.
\end{proof}


If $\textup{l}_o(n):=$ the number of partitions of $n$ in which the number of appearances of the largest part is odd, Then we have the following corollary.
\begin{corollary}
$$\textup{l}_o(n)=\textup{s}(n),$$
where $\textup{s}(n):=$ \textit{number of partitions of $n$ in which the smallest part is odd.}

\end{corollary}

\begin{proof}
By simple combinatorics, the generating function of $\textup{l}_o(n)$ is, 
\begin{align*}
\sum_{n=1}^{\infty}\textup{l}_o(n)q^n = \sum_{n=1}^{\infty}\frac{q^n}{(q)_{n-1}(1-q^{2n})} 
= \sum_{n=1}^{\infty}\frac{q^n}{(q)_{n}(1+q^n)}.
\end{align*}
Hence from (3.23) we have 
$$\sum_{n=1}^{\infty}\textup{l}_o(n)q^n=\sum_{n=1}^{\infty}\textup{s}(n)q^n.$$
By comparing the coefficients of $q^{n}$, we obtain the corollary.
\end{proof}

%
%

\begin{proof}[Proof of Theorem \textup{\ref{genAGL}}]
Put $r=1$ in Lemma \ref{lemmanew}, so as to get
\begin{align}\label{aglgenproof1}
\sum_{n=0}^{\infty}g_n\left(a_{n}-a\right)=\lim_{z\rightarrow 1^-}\frac{d}{dz}(1-z)\left(a f_g- f_{a g} \right).
\end{align}

If we take $$g_n =\frac{1}{(\beta q)_n(q)_n}~and~ a_n= (\alpha)_n,$$
then
\begin{align*}
\frac{d}{dz}(1-z)f_{g}(z)=\frac{d}{dz}(1-z)\sum_{n=0}^{\infty}\frac{1}{(\beta q)_n(q)_n}z^n.
\end{align*}

Using the following version of Fine's identity \cite[Equation (20.2)]{Fine}, we get
\begin{align}\label{aglgenproof2}
\frac{d}{dz}&(1-z)\sum_{n=0}^{\infty}\frac{1}{(\beta q)_n(q)_n}z^n \nonumber \\
&=\frac{d}{dz}\frac{1}{(\beta q)_\infty(zq)_\infty}\sum_{n=0}^{\infty}\frac{(z)_n}{(q)_n}(-\beta)^nq^{n(n+1)/2} \nonumber\\
&=\frac{1}{(\beta q)_\infty(zq)_\infty}\left(\sum_{n=1}^{\infty}\frac{q^n}{1-zq^n}\sum_{n\geq0}\frac{(z)_n}{(q)_n}(-\beta)^nq^{n(n+1)/2} + \frac{d}{dz}\sum_{n\geq0}\frac{(z)_n}{(q)_n}(-\beta)^nq^{n(n+1)/2} \right) \nonumber \\
&=\frac{1}{(\beta q)_\infty(zq)_\infty}\left(\sum_{n=1}^{\infty}\frac{q^n}{1-zq^n}\sum_{n\geq0}\frac{(z)_n}{(q)_n}(-\beta)^nq^{n(n+1)/2} + \frac{d}{dz}\lim_{t\to 0}\sum_{n\geq0}\frac{(z)_n (1/t)_n}{(zt)_n(q)_n}(\beta qt)^n \right).
\end{align}
Now use the second iterate of Heine's transformation \cite[p. 359, Equation (III.2)]{Gasper} with $a\to \beta q/t, b\to 1/t, c\to \beta q, z\to zt^2$ to obtain
\begin{align}\label{aglgenproof3}
\lim_{t\to 0}\sum_{n\geq0}\frac{(z)_n (1/t)_n}{(zt)_n(q)_n}(\beta qt)^n&= \lim_{t\to 0}\frac{(zt^2)_\infty(\beta q)_\infty}{(\beta qt)_\infty(zt)_\infty}\sum_{n\geq0}\frac{(\beta q/t)_n(1/t)}{(\beta q)_n(q)_n}(zt^2)^n \nonumber \\
&=(\beta q)_\infty\sum_{n\geq0}\frac{q^{n^2}\beta^n z^n}{(\beta q)_n(q)_n}.
\end{align}
Then from \eqref{aglgenproof2} and \eqref{aglgenproof3},
\begin{align*}
\frac{d}{dz}(1-z)&\sum_{n\geq0}\frac{1}{(\beta q)_n(q)_n}z^n \\
&=\frac{1}{(\beta q)_\infty(zq)_\infty}\left(\sum_{n=1}^{\infty}\frac{q^n}{1-zq^n}\sum_{n\geq0}\frac{(z)_n}{(q)_n}(-\beta)^nq^{n(n+1)/2} + (\beta q)_\infty\frac{d}{dz}\sum_{n\geq0}\frac{q^{n^2}\beta^n z^n}{(\beta q)_n(q)_n} \right)\\
&=\frac{1}{(\beta q)_\infty(zq)_\infty}\left(\sum_{n=1}^{\infty}\frac{q^n}{1-zq^n}\sum_{n\geq0}\frac{(z)_n}{(q)_n}(-\beta)^nq^{n(n+1)/2} + (\beta q)_\infty\sum_{n\geq1}\frac{n q^{n^2}\beta^n z^{n-1}}{(\beta q)_n(q)_n} \right).
\end{align*}
Hence letting $z \to 1^-,$ we get
\begin{align}\label{aglgenproof4}
\lim_{z\to 1^-}\frac{d}{dz}\left((1-z)\sum_{n\geq0}\frac{1}{(\beta q)_n(q)_n}z^n \right)&=\frac{1}{(\beta q)_\infty(q)_\infty}\left(\sum_{n=1}^{\infty}\frac{q^n}{1-q^n} + (\beta q)_\infty\sum_{n=1}^{\infty}\frac{n \beta^nq^{n^2}}{(\beta q)_n(q)_n} \right).
\end{align}
Now
\begin{align*}
\frac{d}{dz}(1-z)f_{a g}(z)=\frac{d}{dz}(1-z)\sum_{n=1}^{\infty}\frac{(\alpha)_n}{(\beta q)_n(q)_n}z^n.
\end{align*}
Letting $ a \to 0, b \to \alpha,  c\to \beta q $ in \eqref{heine}, we get 
\begin{align*}
(1-z)\sum_{n=0}^{\infty}\frac{(\alpha)_n}{(\beta q)_n(q)_n}z^n =\frac{(\alpha)_\infty}{(\beta q)_\infty(zq)_\infty}\sum_{n=0}^{\infty}\frac{(\beta q/\alpha)_n(z)_n}{(q)_n}(\alpha)^n.
\end{align*}
Then
\begin{align*}
\frac{d}{dz}(1-z)&\sum_{n=1}^{\infty}\frac{(\alpha)_n}{(\beta q)_n(q)_n}z^n\\
&=\frac{d}{dz}\frac{(\alpha)_\infty}{(\beta q)_\infty(zq)_\infty}\sum_{n=0}^{\infty}\frac{(\beta q/\alpha)_n(z)_n}{(q)_n}(\alpha)^n\\
&=\frac{(\alpha)_\infty}{(\beta q)_\infty (zq)_\infty}\left(\sum_{n=1}^{\infty}\frac{q^n}{1-z q^n}\sum_{n=0}^{\infty}\frac{(\beta q/\alpha)_n(z)_n}{(q)_n}(\alpha)^n + \frac{d}{dz}\sum_{n=0}^{\infty}\frac{(\beta q/\alpha)_n(z)_n}{(q)_n}(\alpha)^n \right).
\end{align*}
Letting $z\to 1^-$ in above, we obtain
\begin{align}\label{aglgenproof5}
\lim_{z\to 1^-}\frac{d}{dz}\left((1-z)\sum_{n=1}^{\infty}\frac{(\alpha)_n}{(\beta q)_n(q)_n}z^n\right)=\frac{(\alpha)_\infty}{(\beta q)_\infty (q)_\infty}\left(\sum_{n=1}^{\infty}\frac{q^n}{1- q^n}- \sum_{n=1}^{\infty}\frac{(\beta q/\alpha)_n}{1-q^n}(\alpha)^n \right).
\end{align}
Hence, using \eqref{aglgenproof4} and \eqref{aglgenproof5} in \eqref{aglgenproof1} we have the result.
\end{proof}

\begin{proof}[Proof of Corollary \textup{\ref{lerchsum}}]
Let $\alpha =0$ in \eqref{genAGLequation}.
\end{proof}

\begin{proof}[Proof of Corollary \textup{\ref{AGla=-q}}]
Upon taking $\alpha=-q$ and $\beta=-1$ in \eqref{genAGLequation}, we have
\begin{align*}
\sum_{n=0}^{\infty}\frac{1}{(q^2;q^2)_n}\left( (-q;q)_n -(-q;q)_\infty \right)&=\frac{(-q,q)_\infty}{(q,q)_\infty}\left[ \sum_{n=1}^{\infty}\frac{n (-1)^n q^{n^2}}{(q^2;q^2)_n}+ \frac{1}{(-q)_\infty}\sum_{n=1}^{\infty} \frac{(1)_n}{1-q^n}\alpha^n \right].
\end{align*}
This simplifies to
\begin{align*}
\sum_{n=0}^{\infty}\left( 1 -(-q^{n+1};q)_\infty \right)(q^{n+1};q)_\infty&= \sum_{n=1}^{\infty}\frac{n (-1)^n q^{n^2}}{(q^2;q^2)_n}= -\sum_{n=1}^{\infty} \frac{q^{n(n+1)/2}}{1-q^n},
\end{align*}
with the second equality resulting from letting $\beta=-1$ in \eqref{lerchsumeqn}.
\end{proof}

\begin{proof}[Proof of Corollary \textup{\ref{AGla=q}}]
Upon taking $\alpha=q$ and $\beta=-1$ in \eqref{genAGLequation}, we obtain
\begin{align*}
\sum_{n=0}^{\infty}\frac{1}{(q^2;q^2)_n}\left( (q;q)_n -(q;q)_\infty \right)&=\frac{(q,q)_\infty}{(q,q)_\infty}\left[ \sum_{n=1}^{\infty}\frac{n (-1)^n q^{n^2}}{(q^2;q^2)_n}+ \frac{1}{(-q)_\infty}\sum_{n=1}^{\infty} \frac{(-1)_n}{1-q^n}\alpha^n \right]
\end{align*}
Now use $\beta=-1$ in \eqref{lerchsumeqn} so that
\begin{align*}
\sum_{n=0}^{\infty}\left( 1 -(q^{n+1};q)_\infty \right)(-q^{n+1};q)_\infty&= \sum_{n=1}^{\infty}\frac{n (-1)^n q^{n^2}}{(q^2;q^2)_n}+ 2\sum_{n=1}^{\infty} \frac{(-q)_n}{1-q^{2n}}q^n \\
&= -\sum_{n=1}^{\infty} \frac{q^{n(n+1)/2}}{1-q^n}+ 2\sum_{n=1}^{\infty} \frac{(-q)_n}{1-q^{2n}}q^n.
\end{align*}
\end{proof}

\begin{proof}[Proof of Corollary \textup{\ref{subtraction}}]
Subtract \eqref{AGla=-qenq} from \eqref{AGla=qeqn}.
\end{proof}

\begin{proof}[Proof of Corollary \textup{\ref{AGLq->q^2case}}]
Replace $q $ by $q^2$ and then  substitute $\alpha=q,~\beta =1/q$ in \eqref{genAGLequation} to obtain
\begin{align*}
\sum_{n=0}^{\infty}\frac{1}{( q;q^2)_n(q^2;q^2)_n}\left( (q;q^2)_n -(q;q^2)_\infty \right)&=\frac{(q;q^2)_\infty}{(q^2;q^2)_\infty} \sum_{n=1}^{\infty}\frac{n q^{n(2n-1)}}{(q;q)_{2n}},\\
\sum_{n=0}^{\infty}\left( \frac{(q;q^2)_n}{(q;q^2)_\infty} -1 \right)(q^{2n+1};q)_\infty &= (q;q^2)_\infty\sum_{n=1}^{\infty}\frac{n q^{n(2n-1)}}{(q;q)_{2n}}, \\
\sum_{n=0}^{\infty}\left( \frac{1}{(q^{2n+1};q^2)_\infty} -1 \right)(q^{2n+1};q)_\infty &= (q;q^2)_\infty\sum_{n=1}^{\infty}\frac{n q^{n(2n-1)}}{(q;q)_{2n}},
\end{align*}
which simplifies to 
\begin{align*}
\sum_{n=0}^{\infty}\left((q^{2n+2};q^2)_\infty -(q^{2n+1};q)_\infty \right) &= (q;q^2)_\infty\sum_{n=1}^{\infty}\frac{n q^{n(2n-1)}}{(q;q)_{2n}}.
\end{align*}
This proves the first equality in \eqref{sumtail}. To get the second replace $q$ by $q^2$ in \eqref{lerchsumeqn} and then let $\beta =1/q.$
\end{proof}

\begin{proof}[Proof of Theorem \textup{\ref{FFWfunctiongen}}]
It can be easily observed that 
\begin{align}\label{proofFFWc1}
\sum_{n=0}^{\infty}FFW_c(n)q^n:&=\sum_{n=0}^\infty\left(\sum_{\pi \in \mathcal{D}_n}  (-c)^{\#(\pi)} s(\pi)\right)q^n = -c \sum_{n=1}^{\infty}nq^n(c q^{n+1})_\infty \nonumber \\
&= -c \sum_{n=1}^{\infty}nq^n\left( \sum_{m=0}^{\infty}\frac{(-cq^{n+1})^mq^{m(m-1)/2}}{(q)_m} \right) \nonumber \\
&=-c\sum_{m=0}^{\infty}\frac{(-c)^mq^{m(m+1)/2}}{(q)_m}\sum_{n=1}^{\infty}nq^{(m+1)n}.
\end{align}
We know that
\begin{align}\label{geom}
\frac{d}{dz}\sum_{n=0}^{\infty}z^n = -\frac{1}{(1-z)^2} \implies \sum_{n=0}^{\infty}nz^{n}=\frac{-z}{(1-z)^2}.
\end{align}
Put $z=q^{m+1}$ in \eqref{geom} and use in \eqref{proofFFWc1} to get
\begin{align*}
\sum_{n=0}^{\infty}FFW_c(n)q^n:&=-c\sum_{m=0}^{\infty}\frac{(-c)^mq^{m(m+1)/2}}{(q)_m}\sum_{n=1}^{\infty}nq^{(m+1)n}\\
&=c\sum_{m=0}^{\infty}\frac{(-c)^mq^{m(m+1)/2}}{(q)_m}\left(\frac{q^{m+1}}{(1-q^{m+1})^2}\right)\\
&=-\sum_{m=1}^{\infty}\frac{(-c)^{m}q^{m(m+1)/2}}{(q)_m(1-q^m)}.
\end{align*}
This proves first equality of Theorem \ref{FFWfunctiongen}.
To prove the second equality of Theorem \ref{FFWfunctiongen}, we  write the above equation in a different form, namely, 
\begin{align*}
\sum_{n=0}^{\infty}FFW_c(n)q^n=-\sum_{m=1}^{\infty}\frac{(-c)^{m}q^{m(m+1)/2}}{(q)_m(1-q^m)} =\lim_{t\to 0}\left(\frac{d}{dz}\sum_{m=1}^{\infty}\frac{(z)_n(cq/t)_nt^n}{(q)_m(zq)_n}\right)_{z=1}.
\end{align*} Using \eqref{heine} in the above identity with $a=q, c\to cq, t =z$ and $b\to t$, we have 
\begin{align}\label{endproof}
\sum_{n=0}^{\infty}FFW_c(n)q^n&=\lim_{t\to 0}\left(\frac{d}{dz}\sum_{m=1}^{\infty}\frac{(z)_n(cq/t)_nt^n}{(q)_m(zq)_n}\right)_{z=1} \nonumber\\
&=\lim_{t\to 0}\left(\frac{d}{dz}\frac{(z)_\infty(cq)_\infty}{(zq)_\infty(t)_\infty}\sum_{m=1}^{\infty}\frac{(t)_n}{(cq)_n}z^n\right)_{z=1} \nonumber \\
&=(cq)_\infty\left(\frac{d}{dz}(1-z)\sum_{m=1}^{\infty}\frac{z^n}{(cq)_n}\right)_{z=1}.
\end{align}
From Fine's identity \cite[Equation (20.41)]{Fine}
\begin{align*}
(1-z)\sum_{m=1}^{\infty}\frac{z^n}{(cq)_n}=\frac{(q)_\infty}{(cq)_\infty(zq)_\infty}\sum_{n=0}^{\infty}\frac{(c)_n(z)_n}{(q)_n}q^n.
\end{align*} 
Then 
\begin{align}\label{final}
\left(\frac{d}{dz}(1-z)\sum_{m=1}^{\infty}\frac{z^n}{(cq)_n}\right)_{z=1}&=\frac{d}{dz} \left( \frac{(q)_\infty}{(cq)_\infty(zq)_\infty}\sum_{n=0}^{\infty}\frac{(c)_n(z)_n}{(q)_n}q^n \right)\Bigg|_{z=1} \nonumber\\
&=\frac{1}{(cq)_\infty}\left(\sum_{n=1}^{\infty}\frac{q^n}{1-q^n} - \sum_{n=1}^{\infty}\frac{(c)_n}{(1-q^n)}q^n\right),
\end{align}
where we used \eqref{zn} and \eqref{1byzq} for simplification. From \eqref{final} and \eqref{endproof}, we get the second equality in theorem.
\end{proof}

\begin{proof}[Proof of Theorem \textup{\ref{crippawalacasewithacoro}}]
Letting $t \to 0$ and $g_{n}=\displaystyle{c^n \binom{k+n-1}{n}}$ in Theorem \ref{a,t,g,N} then use the fact $$g(x)= \frac{1}{(1-c x)^k},$$
to get the result.
\end{proof}

Let us conclude this section by giving   finite analogues of some of the corollaries stated in \cite{andrewsfreitas}.
\begin{corollary}
Let $N$ be a positive integer and $|q|<1,$
\begin{align}
\sum_{n=0}^{\infty}\frac{(q)_{\infty}}{(q)_{n}}\left[(q^{n+1})_{N}-1\right] = \sum_{n=1}^{\infty}\left[\begin{matrix} N\\n\end{matrix}\right]_{q}(-1)^nq^{\frac{n(n+1)}{2}}\frac{(q)_{n}}{(1-q^n)}.
\end{align}
\end{corollary}
\begin{proof}
Let $g_n =\displaystyle{\frac{(q)_\infty}{(q)_n}}$, $a=q$  in Theorem \ref{a,t,g,N} and on letting $t \to 0$ on both sides, we get the result.
\end{proof}
Upon taking limit $N \to \infty$ we get \cite[Corollary 4.3 (iv)]{andrewsfreitas}.
\begin{corollary}
Let $N$ be a positive integer and $|q|<1,$
\begin{align}
\sum_{n=0}^{\infty}\left\{ \left(\frac{(q^{N+1})_n}{(q)_{n}}\right)^2 - \left( \frac{1}{(q)_N}\right)^2 \right\} = \frac{1}{(q)^2_{N}}\sum_{n=1}^{\infty}\left[\begin{matrix} N\\n\end{matrix}\right]_{q}\frac{(-1)^nq^{\frac{n(n+1)}{2}}}{(1-q^n)}\left(\frac{(q)_{N}}{(q^{n+1})_{N}}+1\right),
\end{align}
\end{corollary}
\begin{proof}
Put $g_{n}=\frac{(q^{N+1})_n}{(q)_{n}} + \frac{1}{(q)_N}$, $a=q$ in Theorem \ref{a,t,g,N} and let $t\to 0$ on both sides.
\end{proof}
Upon taking limit $N \to \infty$ we get \cite[Corollary 4.3 (vii)]{andrewsfreitas}.
\begin{corollary}
Let $N$ be a positive integer and $|q|<1, $
\begin{align}\label{qnbyn}
(\mathrm{i})\quad\frac{1}{(q)_N(q)_{\infty}}\sum_{n=0}^{\infty}\left[\frac{(q)_n}{(q^{N+1})_n}-(q)_N\right]\frac{(-1)^n q^{\frac{n(n+1)}{2}}}{(q)_n}=\sum_{n=1}^{\infty}\left[\begin{matrix} N+n-1\\n\end{matrix}\right]_{q}\frac{q^n}{(q)_{n}},
\end{align}
\begin{align}\label{qnn+1by2byqn}
(\mathrm{ii})\quad\frac{(q)_N}{(q)_{\infty}}\sum_{n=0}^{\infty}\left[\frac{(q^{N+1})_n}{(q)_n}-\frac{1}{(q)_N}\right]\frac{(-1)^n q^{\frac{n(n+1)}{2}}}{(q)_n}=\sum_{n=1}^{\infty}\left[\begin{matrix} N+n-1\\n\end{matrix}\right]_{q}\frac{(-1)^nq^{\frac{n(n+1)}{2}}}{(q)_{n}}.
\end{align}
\end{corollary}

\begin{proof}
$(\mathrm{i})$ follow by letting $g_n =\frac{(-1)^n q^{\frac{n(n+1)}{2}}}{(q)_n},~t=q$ and $a=0$  in Theorem \ref{a,t,g,N} and similarly $(\mathrm{ii})$ can be obtained by taking $g_n =\frac{(-1)^n q^{\frac{n(n+1)}{2}}}{(q)_n},$ $a=q$ and $t \to 0$  in Theorem \ref{a,t,g,N}.
\end{proof}

In above, \eqref{qnbyn} is finite analogue of identity \cite[Corollary 4.3 (ix)]{andrewsfreitas}, letting $N\rightarrow \infty$ in \eqref{qnn+1by2byqn}  leads to a beautiful identity, namely,
\begin{align*}
\sum_{n=0}^{\infty}\left[\frac{1}{(q)_n}-\frac{1}{(q)_\infty}\right]\frac{(-1)^n q^{\frac{n(n+1)}{2}}}{(q)_n}=\sum_{n=1}^{\infty}\frac{(-1)^nq^{\frac{n(n+1)}{2}}}{(q)^2_{n}}.
\end{align*}

\section{A Combinatorial proof of Theorem \ref{generaldeltafunction} and Weighted partition identities}\label{Combinatorial}
We will begin this section by giving a combinatorial proof of Theorem \ref{generaldeltafunction} and then  obtain a combinatorial interpretation of some of the results mentioned in introduction.

\begin{proof}[Proof of Theorem \textup{\ref{generaldeltafunction}}]
Let us define the function $$\delta_{-t}(q):= \sum_{n=0}^{\infty}\frac{q^{n^2}}{(tq;q^2)_{n}}= \sum_{n=0}^{\infty}\frac{q^{1+3+...+(2n-1)}}{(1-tq)(1-tq^3)...(1-tq^{2n-1})}.$$
It can be noted from above that the coefficient of $q^{N}$ is counting the number of partitions into odd parts without gap, and
 where the power of $t$ is the \textit{total number of parts in a partition minus those counted without multiplicities.} 

%


The conjugate of such a partition has a unique largest part and such that if a part less than the largest part appears as a part, then it appears twice (since the original partition has parts differing by exactly $2$). Also the power of $t$ is the difference between the largest part and the number of parts counted without multiplicities. Thus the generating function of the conjugate partition is

\begin{align*}
1+\sum_{n=1}^{\infty}t^{n}\frac{q^n}{t}\left(1+\frac{q^{2(n-1)}}{t}\right)\left(1+\frac{q^{2(n-2)}}{t}\right)\left(1+\frac{q^{2(n-3)}}{t}\right)....\left(1+\frac{q^{2.2}}{t}\right)\left(1+\frac{q^{2.1}}{t}\right).
\end{align*}   
%
This completes the proof.
\end{proof}

\begin{proof}[Proof of Corollary \textup{\ref{ramanujansumcombi}}]
\label{proofoframanujancombi}

The identity in \eqref{1.14} is a combinatorial equivalent of \eqref{ramaujansumsigma}.

Ramanujan himself gave another series representation for $\sigma(q)$ in \cite[p.~14]{Ramanujannarosa} which is,
\begin{align}\label{ramanujansigma}
\sigma(q)&= \sum_{n=0}^{\infty}\frac{q^{n(n+1)/2}}{(-q)_n} =1+ \sum_{n=1}^{\infty}(-1)^{n-1}q^n(q)_{n-1}.
\end{align}
Then from \eqref{ramanujansigma}, we have
\begin{align*}
\sum_{n=0}^{\infty}\frac{q^{n(n+1)/2}}{(-q)_n} = 1+ \sum_{n=1}^{\infty}\left(\sum_{\pi \in \mathcal{D}(n)} (-1)^{l(\pi) - \#(\pi)}\right)q^n=1+ \sum_{n=1}^{\infty}\left(\sum_{\pi \in \mathcal{D}(n)} (-1)^{\textup{rank}(\pi)}\right)q^n.
\end{align*}
Let us investigate the right-hand side of \eqref{ramaujansumsigma} first.
$\displaystyle{\frac{q^n}{1+q^n}(q^{n+1})_\infty}$ generates a partition of certain number whose smallest part is $n$ which may or may not repeat, and weighted with $(-1)^{\#(\pi)+1}.$ The parts which are greater than the smallest part will be distinct if they appear.
Hence by taking the sum over $n$ we have      
\begin{align*}
\sum_{n=1}^{\infty}\frac{q^n}{1+q^n}(q^{n+1})_\infty= 1+\sum_{n=1}^{\infty}\left(\sum_{\pi \in \mathcal{B}(n)}(-1)^{\#(\pi) +1}\right)q^n.
\end{align*}
In \eqref{ramaujansumsigma}, $(q)_{\infty}$ is counting the number of partitions of a number into distinct parts weighted with $(-1)^{\#(\pi)}$. Thus combining all the observations noted above, we complete the proof of the corollary. 
\end{proof}

In Theorem \ref{newramanujan}$(i)$, we have given a new representation for Ramanujan sum $\sigma(q).$ We now give the weighted partition identity resulting from it.
\begin{corollary}\label{newramanujancoro}
If $\mathcal{D}(n)$, $\mathcal{B}(n)$ are defined as above in introduction, $d(n):=$ number of partition of $n$ into distinct parts. Define $\sigma'(0):=1/2 $ and for $k \geq 1$, $\sigma'(k):= \sum_{d|k}(-1)^{d-1}$. Then for $n \geq 1$
\begin{align*}
\sum_{\pi \in \mathcal{D}(n)}(-1)^{\textup{rank}(\pi)} -d(n) = 2 \left(\sum_{k=0}^{n-1} d(k)\sigma'(n-k) -\sum_{\pi \in \mathcal{B}(n)}1 \right).
\end{align*}
\end{corollary} 
\begin{proof}
It is known \cite{andrewsV} that $\sigma(q)$ counts the partition into distinct parts with weight $ (-1)^{\textup{rank}(\pi)}.$ Now
\begin{align*}
(q;q)_\infty\sum_{n=0}^{\infty}\frac{q^n}{1+q^n}= \sum_{n=0}^{\infty}\left(\sum_{k=0}^{n}d(k)\sigma'(n-k)\right)q^n,
\end{align*}
where $d(k), \sigma'(k)$ are defined in statement of the corollary. Series on the extreme left is generating the weighted identity, which is sum over $\mathcal{B}(n):$ set of partitions of $n$, in which only smallest part is allowed to repeat, rest of the part will be distinct if it appears, weighted with $1.$ By these information we will get our desired result. Hence proof of the corollary completes.
\end{proof}

We will conclude this section by giving weighted partition identity associated to \eqref{oneparameterzagier1}.

\begin{corollary}\label{AGLcombi}
If $\mathcal{D}_k(n),$ $\mathcal{B}(n)$ defined as above, $\#(\pi):$ number of parts in a partition $\pi$ and  $\#s(\pi):$ number of smallest parts appearing in a partition $\pi,$ then we have
\end{corollary}
\begin{align*}
\sum_{k=0}^{n}c^k\left(\sum_{\pi \in \mathcal{D}_k(n)}(-1)^{\#(\pi) +1}\right) = \sum_{\pi \in \mathcal{B}(n)}(-c)^{\#s(\pi)-1}.
\end{align*}
\begin{proof}
Take right hand side of \eqref{oneparameterzagier1} with $t=q,$
\begin{align*}
\sum_{n=1}^{\infty}\frac{(q)_{\infty}}{(1-c q^n)(q)_n}q^n = \sum_{n=1}^{\infty}\frac{(q^{n+1})_{\infty}}{(1-c q^n)}q^n
\end{align*}
if we take $n$ as a smallest part in a partition then $(q^{n+1})_\infty$ gives the partition of an integer into distinct parts  in which the smallest part is $n$ and $\frac{q^n}{1- cq^n}$ gives the partition of an integer in which the smallest part is $n$ and power of $c$ counts the number of appearances of smallest part $n$ in a partition minus 1.\\
Also the left-hand side of \eqref{oneparameterzagier1} $\left((q)_m -(q)_\infty\right),$ counts the partition of a certain integer into distinct part in which parts are strictly greater than $m$. By summing over $m$ the proof follows.
\end{proof}

\section{Concluding Remarks}
Letting $N \to \infty $ and $t=q$ in \eqref{zagier}, we get
\begin{align}\label{newzagier}
\sum_{n=0}^{\infty}c^n \left[ (q)_n - (q)_\infty \right]&=\sum_{n=1}^{\infty}(1+c+c^2+... c^{n-1})(q)_{n-1}q^{n}.
\end{align} 
The special case $c=1$ of the above identity is a result of Zagier \cite[Equations (17), (18)]{zagier}. Zagier \cite{zagier} uses his result in the proof of his Theorem $2$. It will be interesting to see if our identity \eqref{newzagier} could be used to generalize his Theorem $2$.

Do there exist identities similar to those in Corollary \ref{mocktheta} for  other mock theta functions of order three? If so, it will be interesting to study their combinatorial interpretations.

\section{Acknowledgements}
I would like to thank Professor Atul Dixit for his constant support throughout the research. I would also like to take opportunity to thank my institution Indian Institute of Technology Gandhinagar for providing me state-of-the-art research facilities.

\end{document}